%% file: draft.tex
\documentclass{amsart}
\usepackage{geometry}

\usepackage{amsfonts,amssymb, amsthm}
\usepackage{amsmath}
\usepackage{mathrsfs}
\usepackage{hyperref,graphicx}
\hypersetup{
    colorlinks=true,
    linkcolor=blue,
    filecolor=magenta,      
    urlcolor=cyan,
    citecolor=magenta}
\usepackage{tikz}

\usepackage{xcolor}

\definecolor{asparagus}{rgb}{0.53, 0.66, 0.42}

\newcommand{\ZZ}{\mathbb Z}
\newcommand{\FF}{\mathbb F}
\newcommand{\bo}{b{\rm o}}

\newcommand{\tmf}{\operatorname{tmf}}

\newcommand{\sm}{\wedge}
\newcommand{\sq}[1]{Sq^{#1}}

\newcommand{\Ext}{\operatorname{Ext}}
\newcommand{\Hom}{\operatorname{Hom}}
\newcommand{\A}{\mathscr A}
\newcommand{\st}{\>:\>}
\newcommand{\alphas}{\alpha_1,
\alpha_2,\ldots}
\newcommand{\zeroalphas}
{0,\alpha_2,\alpha_3,\ldots}
\newcommand{\epsilons}
{\ve_2,\ldots,\ve_{m+2}}
\newcommand{\zeros}{0,0,\ldots,0}
\newcommand{\floor}[1]{\lfloor #1\rfloor}

\newcommand{\ve}{\varepsilon}
\newcommand{\abs}[1]{\left|#1\right|}

\newtheorem{theorem}{Theorem}[section]
\newtheorem{lemma}[theorem]{Lemma}
\newtheorem{corollary}[theorem]{Corollary}

\theoremstyle{definition}
\newtheorem{example}{Example}[section]
\newtheorem{remark}[example]{Remark}

\author{Scott M. Bailey 
\and Donald
M. Larson
}
\title{Enumerating partitions arising in homotopy theory}
\date{\today}

\begin{document}

\begin{abstract}
We present an infinite family of recursive formulas that count binary integer partitions satisfying natural divisibility conditions and show that these counts are interrelated via partial sums.  Moreover, we interpret the partitions we study in the language of graded polynomial rings and apply this to the mod $2$ Steenrod algebra to compute the free rank of certain homology modules in stable homotopy theory.


\vspace{10pt}


\noindent{\bf Mathematics
Subject Classification 2020.}
\href{
https://mathscinet.ams.org/mathscinet/msc/msc2020.html?t=55Sxx&btn=Current}{55S10}, 
\href{
https://mathscinet.ams.org/mathscinet/msc/msc2020.html?t=16Wxx&btn=Current}{16W50}, 
\href{
https://mathscinet.ams.org/mathscinet/msc/msc2020.html?t=55Txx&btn=Current}{55T15}, 
\href{
https://mathscinet.ams.org/mathscinet/msc/msc2020.html?t=05Axx&btn=Current}{05A17}
\end{abstract}
\maketitle

\section{Introduction} 
\label{sec:intro}


A binary partition of an integer $n\geq0$ is a partition of $n$ whose parts are powers of 2.  Such partitions have the form $n=\alpha_1\cdot2^0+\alpha_2\cdot2^1+\cdots+\alpha_k\cdot2^{k-1}$ for $0\leq\alpha_i\in\ZZ$, and it is well-known how to count them recursively. Indeed, if $r_{-1}(n)$ denotes the number of binary partitions of $n$, then $r_{-1}(n)=r_{-1}(n-1)$ for $n$ odd and $r_{-1}(n)=r_{-1}(n-1)+r_{-1}(n/2)$ for $n$ even, where $r_{-1}(0)=1$. The sequence $r_{-1}(n)$ has as its first few terms
\[
1,1,2, 2, 4, 4, 6, 6, 10, 10, 14, 14, 20, 20, 26, 26,\ldots
\]
and appears as entry \texttt{A018819} in the On-Line Encyclopedia of Integer Sequences (OEIS) \cite{OEIS}. The subscript $-1$ is a notational convenience, as the sequence $r_{-1}(n)$ will be subsumed into an infinite family of sequences $\{r_m(n):m\geq-1\}$ whose provenance we shall now begin to describe.

The conspicuous two-fold repetition in the sequence $\{r_{-1}(n)\}$ disappears if we instead count binary partitions of $n$ for which $2\mid\alpha_1$ (that is to say, binary partitions with an even number of 1s).  If $r_0(n)$ denotes the number of binary partitions with the aforementioned divisibility condition on $\alpha_1$, then $r_0(n)=r_0(n-2)+r_0((n-2)/2)$ for $n\equiv2\mod4$ and $r_0(n)=r_0(n-2)+r_0(n/2)$ for $n\equiv0\mod4$, where $r_0(0)=1$.
Ignoring the obvious zeros occurring when $n$ is odd yields
the
sequence $r_0(2n)$
which has as its first few terms
\[
1, 2, 4, 6, 10, 14, 20, 26, 36, 
46, 60, 74, 94, 114, 140, 166,\ldots
\]
and appears as entry
\texttt{A000123} of the OEIS, where it is also observed that $r_0(2n)$ is the sum of $r_{-1}(i)$ from $i=0$ to $i=n$.

A desire to perpetuate the partial sum relationship would dictate setting $r_1(n)$ equal to the number of binary partitions of $n$ for which $4\mid\alpha_1$ and $2\mid\alpha_2$ (see entry \texttt{A131205} in the OEIS).  In fact, the entire infinite family of counts that is the subject of this paper emerges in precisely this fashion.  Let $r_m(n)$ denote the number of binary partitions of $n$ for which the $m+1$ divisibility conditions
\begin{equation}\label{eq:divisibility}
2^{m+2-i}\mid\alpha_i,\quad
1\leq i\leq m+1
\end{equation}
hold.  Our main result is a recursive formula for $r_m(n)$.


\begin{theorem}\label{thm:main}
For $0<n\equiv2^{m+1}\mod2^{m+2}$,
we have the recursive formula
\[
r_m(n)=r_m(n-2^{m+1})+\sum_{j=0}^{\floor{m/2}} \binom {m+1}{2j+1} r_m\left(\dfrac{n-(2j+1)2^{m+1}}2
\right).
\]
Similarly, for $0\leq n\equiv0\mod 2^{m+2}$,
\[
r_m(n)=r_m(n-2^{m+1})+\sum_{j=0}^{\floor{(m+1)/2}}\binom {m+1}{2j} r_m\left(\dfrac{n-2j\cdot2^{m+1}}2\right).
\]
For all other $n$, $r_m(n)=0$.
\end{theorem}

Furthermore, the sequences $r_m(n)$ are all related to each other via partial sums.

\begin{theorem}\label{thm:main_partial}
For $m\geq0$ and $n\equiv0\mod 2^{m+1}$,
$\displaystyle r_m(n)=\sum_{k=0}^{
n/2^{m+1}}r_{m-1}(2^{m}k)$.
\end{theorem}

Theorem \ref{thm:main} can be framed in terms of the combinatorics of graded polynomial rings.  Let
\[
H(m)=\FF_2[X_1^{2^{m+1}},X_2^{2^m},X_3^{2^{m-1}},\ldots,X_{m+1}^2,X_{m+2},X_{m+3},\ldots]
\]
where the indeterminate $X_i$ has grading
$|X_i|=2^{i-1}$.
Then a monomial $x\in H(m)$ of grading $n$ of the form $x=X_1^{\alpha_1}X_2^{\alpha_2}\cdots X_k^{\alpha_k}$ is precisely the same data as a partition $n= \alpha_1\cdot2^0+\alpha_2\cdot 2^1+\cdots+\alpha_k\cdot 2^{k-1}$ of $n$ satisfying the divisibility conditions \eqref{eq:divisibility}, yielding the following corollary.

\begin{corollary}\label{cor:polyring}
The number of monomials $x\in H(m)$ with grading $n$ is $r_m(n)$.
\end{corollary}

While of possible independent interest in number theory and theoretical computer science, the quantities $r_m(n)$ and their interpretation in Corollary \ref{cor:polyring} also have a natural application in stable homotopy theory to splittings of the homology of Brown-Gitler spectra.  

\begin{theorem}\label{thm:bgspectra}
Let $0\leq j\in\ZZ$.  As a module over the subalgebra $\A(1)$ of the mod 2 Steenrod algebra $\A$ generated by $Sq^1$ and $Sq^2$, the mod 2 homology of the $j$th integral Brown-Gitler spectrum $B_1(j)$ has free rank $(r_1(4j)-b(j))/8$ where 
\begin{equation}\label{eq:BJ}
b(j) = \begin{cases}
    4j - 2\alpha(j) + 1, &\text{if }\alpha(j)\equiv 0,1\mod{4},\\
    4j - 2\alpha(j) + 5, &\text{if }\alpha(j)\equiv 2,3\mod{4}
\end{cases}
\end{equation}
and where $\alpha(j)$ is number of $1s$ in the binary expansion of $j$.
\end{theorem}
(Compare with Theorem \ref{Thm:Rank1}.)

Section~\ref{sec:application_to_homotopy_theory} is devoted to a discussion of Theorem \ref{thm:bgspectra} and its proof, including all the requisite definitions and a sketch of the relevant background material. Sections \ref{sec:preliminary_results}, \ref{sec:proof_main_theorem}, and \ref{sec:partial_sum} comprise the technical heart of the paper.
Section~\ref{sec:preliminary_results} gives results required for the proof of Theorem~\ref{thm:main} in Section~\ref{sec:proof_main_theorem}.  Section~\ref{sec:partial_sum} contains the proof of the partial sums result given by Theorem~\ref{thm:main_partial}.

\subsection*{Acknowledgements}  The authors would like to thank Paul Goerss, Mark Behrens, and the anonymous referee for useful comments and suggestions.

\section{The homology of Brown-Gitler spectra}\label{sec:application_to_homotopy_theory}
Originally introduced for the purpose of studying obstructions to immersions of manifolds, Brown-Gitler spectra are geometric objects tailored to have (co)homology modules with certain properties.  Among the surprisingly numerous results in stable homotopy theory they have spawned are the wedge sum splittings of spectra related to connective real $K$-theory found by Mahowald \cite{Mah81} and the first author \cite{Bai10}.  In this section, we apply Corollary \ref{cor:polyring} to compute the free rank of the homology (with coefficients in $\FF_2$) of a certain family of Brown-Gitler spectra as modules over $\A(1)$, the submodule of the mod 2 Steenrod algebra $\A$ generated by $Sq^1$ and $Sq^2$.  We avail ourselves of the language of stable homotopy theory to describe our result: spectra and the cohomology theories they represent, completion of spectra at a prime $p$, the Steenrod algebra, and the Adams spectral sequence.  Readers unfamiliar with this material may consult the Chicago lectures of Adams \cite{Ada74}, the guide to computations in homotopy theory at the prime 2 by Beaudry and Campbell \cite{BC18}, or the survey of connective real $K$-theory by Bruner and Greenlees \cite{BG10}.  Readers unconcerned with homotopy theoretic applications may safely skip to Section \ref{sec:preliminary_results}.

Throughout this section, all spectra are implicitly completed at the prime 2, and $[X,Y]$ denotes the graded abelian group of homotopy classes of maps from the spectrum $X$ to the spectrum $Y$.

\subsection{The impetus from homotopy theory} \label{subsec:impetus_homotopy} The mod 2 cohomology $H^*X$ of a spectrum $X$ is represented by the mod 2 Eilenberg-Mac Lane spectrum $H\FF_2$, i.e.,  $H^*X= [X, H\FF_2]$. The mod $2$ Steenrod algebra $\A=[H\FF_2, H\FF_2]$ is the algebra of (stable) cohomology operations.  As an algebra over $\FF_2$, $\A$ is generated by $\sq{i}$ for $i \geq 0$ which act by post-composition. Modules over $\A$ have been well-studied in the literature, especially in light of the Adams spectral sequence
\[
    \Ext_\A(H^*Y, H^* X) \Rightarrow [X,Y] \otimes \ZZ_2
\]
which approximates the stable homotopy groups of spheres when $X=Y=S$, the sphere spectrum.  

Given the subalgebra $\A(m)$ of $\A$ generated by $\{\sq{1},\sq{2},\dots,\sq{2^m}\}$, consider the $\A(m)$-module $\A \otimes_{\A(m)} \FF_2$, where the right action of $\A(m)$ on $\A$ is induced by the inclusion and the left action of $\A(m)$ on $\FF_2$ is induced by $\sq{0} \mapsto 1$ and $\sq{i} \mapsto 0$ for $i>0$.
The inclusions $\A(m)\to \A(m+1)$ give rise to an infinite tower of surjections
\begin{equation}\label{Eqn:Tower}
    \A \to \A \otimes_{\A(0)} \FF_2 \to \A \otimes_{\A(1)} \FF_2 \to \A \otimes_{\A(2)} \FF_2 \to \A \otimes_{\A(3)} \FF_2 \to \cdots.
\end{equation}
It will be convenient to put $\A(-1) = \{ \sq{0} \} = \FF_2$, so that $\A \cong \A \otimes_{\A(-1)} \FF_2$. Noting that $\A \cong H^* H\FF_2$, one can ask if the remaining modules in \eqref{Eqn:Tower} can be realized as the mod 2 cohomology of some spectrum. The case $m \geq 3$ requires the existence of a non-trivial map of spheres which has been shown not to exist. However, it is well-known that such spectra exist for $0\leq n\leq2$, namely
\begin{align*}
    H^*H\ZZ &\cong \A \otimes_{\A(0)} \FF_2, \\ 
    H^*\bo &\cong \A \otimes_{\A(1)} \FF_2, \\ 
    H^*\tmf &\cong \A \otimes_{\A(2)} \FF_2,
\end{align*}
where $H\ZZ$ is the integral Eilenberg-Mac Lane spectrum, $\bo$ is the connective real $K$-theory spectrum, and $\tmf$ is
the connective spectrum of  topological modular forms---an analog of $K$-theory with connections to modular forms in number theory (see \cite{Larson} for an instance of such a connection away from the prime 2).  Indeed, these spectra and their associated cohomology theories have become essential for computations inside the stable homotopy groups of spheres, and are interesting in their own right.

A change-of-rings theorem yields 
\[
    \Ext^{s,t}_\A(\A\otimes_{\A(m)} \FF_2, H^* X) \cong \Ext^{s,t}_{\A(m)}(\FF_2,H^*X).
\]
In particular, the $E_2$-term approximating the integral cohomology of $X$ can be determined by an understanding of the $\A(0)$-module structure of $H^* X$. Similarly, the $E_2$-terms approximating the connective real $K$-theory of $X$ and the $\tmf$-cohomology of $X$ can be determined by understanding the $\A(1)$- and $\A(2)$-module structures of $H^*X$, respectively. 

Modules over $\A(1)$ have been fully classified up to stable $\A(1)$ isomorphism \cite[Section 3]{AP76}. Two $R$-modules $M_1$ and $M_2$  are stably $R$ isomorphic if there exist free $R$-modules $F_1$ and $F_2$ such that $M_1 \oplus F_1 \cong M_2 \oplus F_2$. The original impetus of the current paper was a better understanding of the free $\A(1)$-module summands of $\A \otimes_{\A(m)} \FF_2$.  Such an understanding would ultimately provide a clearer picture of important computations currently in the literature, as we shall discuss in the following subsections.  See for example \cite{Dav87} and \cite{BBBCX}.

\subsection{Connection with binary partitions}

It is often convenient to perform computations in the dual $\A_* = \Hom_{\FF_2}(\A,\FF_2)$.  Indeed, while $\A$ is a graded noncommutative algebra with many relations, its dual $\A_* \cong \FF_2[\xi_1,\xi_2, \xi_3,\dots]$ is a polynomial ring with grading $\abs{\xi_i} = 2^{i} - 1$. Furthermore,
\begin{align*}
    H_*H\FF_2 &\cong (\A\otimes_{\A(-1)} \FF_2)_* \cong \A_*,\\ 
    H_*H\ZZ &\cong (\A \otimes_{\A(0)} \FF_2)_* \cong \FF_2[\xi_1^2, \xi_2,\xi_3, \xi_4,\dots], \\ 
    H_*\bo &\cong (\A \otimes_{\A(1)} \FF_2)_* \cong \FF_2[\xi_1^4, \xi_2^2, \xi_3, \xi_4, \dots], \\ 
    H_*\tmf &\cong (\A \otimes_{\A(2)} \FF_2)_* \cong \FF_2[\xi_1^8,\xi_2^4,\xi_3^2,\xi_4,\dots]
\end{align*}
(see \cite{Mil58, Koc81, Mah81, Rez07}). Upon replacing the grading mentioned above by the alternative grading convention $|\xi_i|=2^{i-1}$
(hereafter called a ``weight'' and denoted $\omega(\xi_i)=2^{i-1}$ following the language and notation of
\cite[\S3]{Mah81})
the reader should observe that $H(m) \cong (\A\otimes_{\A(m)}\FF_2)_*$.  Hence, Theorem~\ref{thm:main} and Corollary \ref{cor:polyring} together provide a recursive formula for the number of generators of $(\A \otimes_{\A(m)}\FF_2)_*$ of a given weight.

\begin{example}
The number of monomials of weight $n$ in $H_*H\FF_2$ and the number of monomials of weight $2n$ in $H_*H\ZZ$ are counted by the sequences $r_{-1}(n)$ and $r_0(2n)$, respectively, exhibited at the start of Section \ref{sec:intro}. Moreover,
\[
r_1(n)=\begin{cases}
1,&n=0,\\
r_1(n-4)+2r_1((n-4)/2),
&0<n\equiv4\mod 8,\\
r_1(n-4)+r_1(n/2)+
r_1((n-8)/2),&0<n\equiv0\mod8,\\
0,&\text{otherwise}.
\end{cases}
\]
and so the sequence $r_1(4n)$ counting the number of monomials of weight $4n$ in $H_*\bo$ has
\[
1, 3, 7, 13, 23, 37, 57, 83, 119, 165, 225, 299, 393, 507, 647, 813,
\ldots
\]
as its first few terms.
\end{example}




\begin{example}
Since
\[
r_2(n)=\begin{cases}
1,&n=0,\\
r_2(n-8)+3r_2((n-8)/2)+
r_2((n-24)/2),&0<n\equiv8\mod16,\\
r_2(n-8)+r_2(n/2)+3r_2((n-16)/2),
&0<n\equiv0\mod16,\\
0,&\text{otherwise},
\end{cases}
\]
it follows that 
the sequence $r_2(8n)$ counting the number of 
monomials of weight $8n$ in
$H_*\tmf$ has
\[
1,4,11,24,47,84,141,224,
343,508,733,1032,1425,1932,2579,
3392,\ldots
\]
as its first few terms.
\end{example}

\begin{example} 
As noted above, there does not exist a
spectrum with mod
2 homology isomorphic to
$(\A \otimes_{\A(m)} \FF_2)_*$ for $m\geq3$.  
Theorem \ref{thm:main}
implies
\[
r_3(n)=\begin{cases}
1,&n=0,\\
r_3(n-16)+4r_3((n-16)/2)
+4r_3((n-48)/2),
&0<n\equiv16\mod32,\\
r_3(n-16)+r_3(n/2)+6r_3((n-32)/2)
+r_3((n-64)/2),
&0<n\equiv0\mod32,\\
0,&\text{otherwise}.
\end{cases}
\]
and so $H(3)$, the first of the polynomial algebras $H(m)$ not realizable as the mod 2 homology of a spectrum, has $r_3(16n)$ monomials of weight $16n$,
where $r_3(16n)$ has
as its first few terms
\[
1,5,16,40,87,171,312,
536,879,1387, 
2120,3152,4577,6509,9088,12480,\ldots
\]
\end{example}

One can observe the partial sum relationships between the sequences in the above examples given by Theorem~\ref{thm:main_partial}.

\subsection{Integral Brown-Gitler spectra}\label{subsec:integralBGspectra}

Brown and Gitler \cite{BG73} constructed a family $\{B(j)\>\vert\>j \geq 0\}$ of $\ZZ/2\ZZ$-complete spectra whose cohomology algebras are cyclic as modules over $\A$. For a given $j$, the generator $\alpha : B(j) \to H\FF_2$ gives rise to a surjection $\alpha_* : B(j)_k X \to (H\FF_2)_k X$ in homology for $k < 2j + 2$ and any CW-complex $X$, resulting in these Brown-Gitler spectra having many applications in homotopy theory. This led Goerss, Jones, and Mahowald \cite{GJM86} to construct  analogous families of Brown-Gitler spectra over $H\ZZ$ and $\bo$ (as well as $BP\langle 1 \rangle$).  For the purposes of this paper, we will denote the Brown-Gitler, integral Brown-Gitler, and $\bo$ Brown-Gitler spectra by $B_0(j)$, $B_1(j)$, and $B_2(j)$, respectively.  With the weight $\omega$ defined above, the homology of these Brown-Gitler spectra are given by
\begin{equation}\label{eq:homologyBG}
    H_* B_i(j) \cong \{ x \in (\A\otimes_{\A(i-1)} \FF_2)_* \>\vert\> \omega(x) \leq j2^i\}
\end{equation}
where $(\A\otimes_{\A(i-1)} \FF_2)_*$ denotes the $\FF_2$-dual of $(\A\otimes_{\A(i-1)} \FF_2)$.

The stable $\A(1)$-isomorphism classes of $H_*B_1(j)$ are related to the family of Milgram modules \cite{Mil75} which Davis, Gitler, and Mahowald \cite{DGM81} denote $Q_{i,n}$ for $i\geq0$ and $n\in\{0,1,2,3\}$. These modules are displayed in Figure~\ref{fig:Milgram}.  The circles represent generators of $\FF_2$ in the degree indicated by the column, while straight and curved lines represent an action of $\sq{1}$ and $\sq{2}$, respectively.

\begin{figure}
    \centering
    \scalebox{.9}{\input{milgram}}
    \caption{Milgram modules $Q_{i,n}$ ($0\leq i\leq3$)}
    \label{fig:Milgram}
\end{figure}
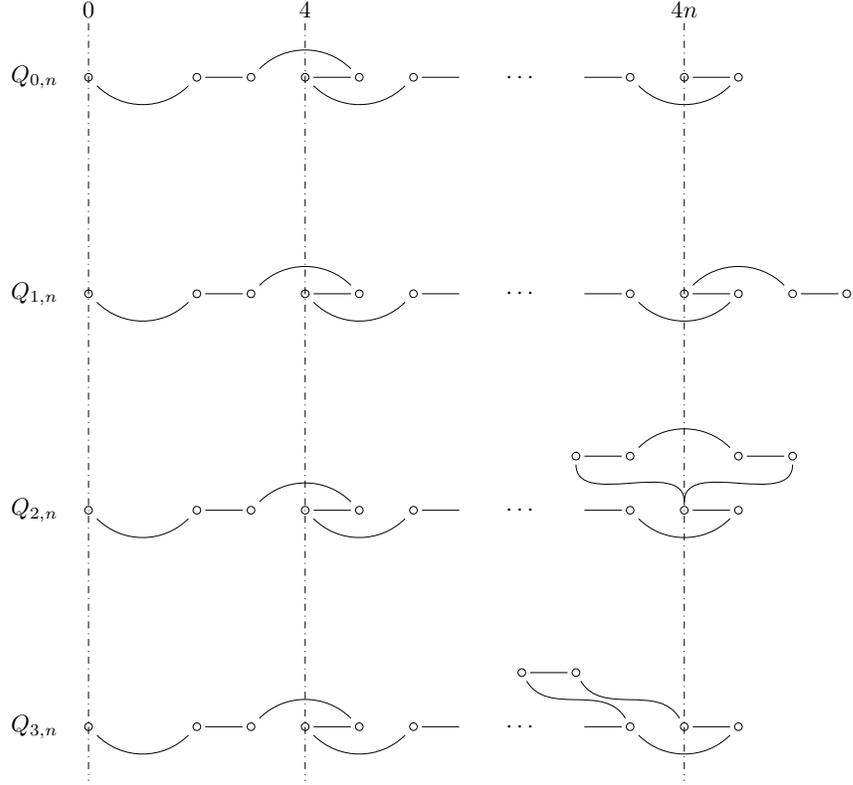

\subsection{Rank of free $\A(1)$-module summands}

\begin{lemma}[Lemma 3.12 \cite{DGM81}]\label{lem:stable-iso-Milgram}
There is an isomorphism $H_* B_1(j) \cong F_{1,j} \oplus Q_{\alpha(j),j-D}$ of $\A(1)$-modules, where
\[
D = \begin{cases}
        2\ell, &\text{if } \alpha(j) = 4\ell,\\
        2\ell+1, &\text{if } 4\ell + 1 \leq \alpha(j) \leq 4\ell+3,
\end{cases}
\]
$F_{1,j}$ is a free $\A(1)$-module, and the first subscript $\alpha(j)$ of $Q_{\alpha(j),j-D}$ is taken modulo 4.
\end{lemma}


An application of Theorem~\ref{thm:main} therefore determines the free rank of $F_{1,j}$, yielding the following theorem (which is merely a restatement of Theorem \ref{thm:bgspectra}).
\begin{theorem}\label{Thm:Rank1}
The rank of $F_{1,j}$ as an $\A(1)$-module is $f_{1,j} = \dfrac{1}{8}\left(r_1(4j)-b(j)\right)$ where 
\begin{equation}\label{eq:Mj}
b(j) = \begin{cases}
    4j - 2\alpha(j) + 1, &\text{if }\alpha(j)\equiv 0,1\mod{4},\\
    4j - 2\alpha(j) + 5, &\text{if }\alpha(j)\equiv 2,3\mod{4}.
\end{cases}
\end{equation}
\end{theorem}
\begin{proof}
From Figure \ref{fig:Milgram}, we see that
\begin{equation*}
\dim_{\FF_2}(Q_{i,n})=4n+\begin{cases} 1,&i=0,\\ 3&i=1,3,\\ 5,&i=2.\end{cases}
\end{equation*}
Lemma \ref{lem:stable-iso-Milgram} therefore yields
\begin{align*}
\begin{split}
\dim_{\FF_2}(Q_{\alpha(j),j-D})&=4j-4\begin{cases}2\floor{\alpha(j)/4},&\alpha(j)\equiv0\mod4\\2\floor{\alpha(j)/4}+1,&\alpha(j)\not\equiv0\mod4\end{cases}+\begin{cases}1,&\alpha(j)\equiv0\mod4\\3,&\alpha(j)\equiv1,3\mod4\\5,&\alpha(j)\equiv2\mod4\end{cases}\\
&=4j+\begin{cases}-2\alpha(j)+1,&\alpha(j)\equiv0\mod4\\-2(\alpha(j)-1)-1,&\alpha(j)\equiv1\mod4\\-2(\alpha(j)-2)+1,&\alpha(j)\equiv2\mod4\\-2(\alpha(j)-3)-1,&\alpha(j)\equiv3\mod4.\end{cases}
\end{split}
\end{align*}
Note that Theorem~\ref{thm:main} and \eqref{eq:homologyBG} together imply that $r_1(4j)$ is the $\FF_2$-dimension of $H_* B_1(j)$.  By Lemma~\ref{lem:stable-iso-Milgram}, the remaining classes generate $F_{1,j}$.  Since $\A(1)$ has 8 classes, the result follows.
\end{proof}

\begin{remark}
The sequence $\{f_{1,j}\}$ has as its first few terms
\[
0, 0, 0, 0, 1, 2, 4, 7, 11, 16, 23, 32, 43, 57, 74, 95, \dots.
\]
The reader should notice the first appearance of a free $\A(1)$-summand among the homologies of integral Brown-Gitler spectra is in $H_* B_1(4)$. 
\end{remark}

Consider the elements $N_n(\ell) = \{x \in (\A\otimes_{\A(n)}\FF_2)_*\> \vert \> \omega(s) = \ell\}$ of homogeneous weight $\ell$. For $m \leq n$, there is an $\A(m)$-module isomorphism
\begin{equation}\label{eq:Am-module-splitting}
    \left( \A\otimes_{\A(n)}\FF_2\right)_* \cong \bigoplus_{j \geq 0} N_n(j2^{n+1}).
\end{equation}
This is a straightforward extension of \cite[Lemma 2.1]{Mah81} for $m=n=1$ and \cite[Proposition 2.3]{Bai10} for $m=1$ and $n=2$.  The \emph{Verschiebung} is the algebra homomorphism $V:\A_* \to \A_*$ defined on generators by 
\[
    V(\xi_i) = \begin{cases}
            1,  &i=0,1, \\
            \xi_{i-1}, &i \geq 2.\\
    \end{cases}
\]
An extension of the proof of \cite[Proposition 2.3]{Bai10} shows that $V$ induces an $\A(m)$-module isomorphism
\[
N_n(j2^{n+1})\cong \Sigma^{j2^{n+1}}\{x \in (\A \otimes_{\A(n-1)} \FF_2)_* \>\vert\>\omega(x)\leq j 2^n \}.
\]
In particular $\Sigma^{j2^{i+1}}H_* B_i(j) \cong N_i(j2^{i+1})$ and, as a result, Theorem~\ref{Thm:Rank1} can be used to compute the number of free copies of $\A(1)$ in $(\A\otimes_{\A(m)}\FF_2)_*$ for all $m \geq 1$.  An example of this is given by the following theorem of the first author for $m=2$.

\begin{theorem}[Theorem 4.1 \cite{Bai10}]\label{thm:bosmtmf}
There is an isomorphism of graded $\bo_*$-algebras 
\begin{equation}\label{eq:bosmtmf_iso}
\pi_*(\bo\sm\tmf) \cong \dfrac{\bo_*[\sigma, b_i, \mu_i\>\vert\>i \geq 0]}{(\mu b_i^2-8b_{i+1}, \mu b_i -4 \mu_i, \eta b_i)} \oplus F
\end{equation}
where $\abs{\sigma}=8$, $\abs{b_i} = 2^{i+4} -4$, $\abs{\mu_i} = 2^{i+4}$ and $F$ is a direct sum of $\FF_2$ in varying dimensions.
\end{theorem}

The proof of Theorem~\ref{thm:bosmtmf} involves a computation of the Adams $E_2$-term
\[
E_2^{s,t}=\Ext_{\A}^{s,t}(H^*(\bo\sm \tmf), \FF_2) \Rightarrow \pi_{t-s} (\bo \sm \tmf)
\]
which, after applying a change-of-rings theorem and dualizing, becomes $\Ext_{\A(1)}^{s,t}(\FF_2, H_* \tmf)$. By the argument immediately preceding Theorem~\ref{thm:bosmtmf}, $H_* \tmf \cong \bigoplus_{j \geq 0} \bigoplus_{0 \leq i \leq j} \Sigma^{4i+8j} H_* B_1(j)$.  In particular
\begin{align*}
    H_*\tmf &\cong \bigoplus_{j \geq 0} \bigoplus_{0 \leq i \leq j} \Sigma^{4i+8j} \left(F_{1,j} \oplus Q_{\alpha(j),j-D}\right) \\
    &\cong \bigoplus_{j \geq 0} \bigoplus_{0 \leq i \leq j} \Sigma^{4i+8j} Q_{\alpha(j),j-D} \oplus \bigoplus_{j \geq 0} \bigoplus_{0 \leq i \leq j} \Sigma^{4i+8j} F_{1,j} \\
    &\cong \bigoplus_{j \geq 0} \bigoplus_{0 \leq i \leq j} \Sigma^{4i+8j} Q_{\alpha(j),j-D} \oplus \bigoplus_{j\geq 0}F_{2,j}
\end{align*}
where $F_{2,j} = \bigoplus_{0 \leq i \leq j} \Sigma^{4i+8j} F_{1,j}$ can be regarded as the analog of $F_{1,j}$ for $H_* B_2(j)$. With this notation, the module $F$ in Theorem~\ref{thm:bosmtmf} is given by $F = \Ext_{\A(1)}^{s,t}(\bigoplus_{j\geq0} F_{2,j}, \FF_2) \cong \bigoplus_{j \geq 0} \FF_2^{\dim_{\A(1)}(F_{2,j})}$.  If we put $f_{2,j}=\dim_{\A(1)}(F_{2,j})$, then
\[
f_{2,j}=\sum_{i = 0}^{j} \dim_{\A(1)}(F_{1,j})
\]
so that $f_{2,j}$ is the number of free $\A(1)$ summands in $H_* B_2(j)$ and $\{f_{2,j}\}$ is the sequence of partial sums of $\{f_{1,j}\}$. This partial sum pattern can be continued inductively, even though a geometric interpretation of the relevant polynomial rings is lost for $m \geq 3$ as we noted in Subsection \ref{subsec:impetus_homotopy}.

\begin{remark}
As a result of \eqref{eq:Am-module-splitting},
the Verschiebung homomorphism $V:N_{n}(2^{n+1}k)\to\bigoplus_{i=0}^k N_{n-1}(2^ni)$ is an isomorphism of $\A(m)$-modules for $m\leq n-1$. Since $\dim_{\FF_2}(N_{\ell}(k))=r_{\ell}(k)$ for all $k \geq 0$, the Verschiebung gives us another way of viewing Theorem \ref{thm:main_partial}.
\end{remark}

\begin{remark}
While we have provided a method for counting the rank of the free $\A(1)$ summands of $(\A\otimes_{\A(m)}\FF_2)_*$ for $m\geq1$, it would be desirable to determine the degree of the corresponding generators.  This will be the subject of future work.
\end{remark}

\section{Preliminary results on binary partitions}\label{sec:preliminary_results}
Let $\pi=(n;\alphas)$ denote the binary partition $\pi$ of $n$ with $\alpha_1$ the coefficient on $2^0$, $\alpha_2$ the coefficient on $2^1$, etc., and let $P_m(n)$ denote the set of all binary partitions $\pi=(n;\alphas)$ of $n$ satisfying the divisibility conditions given by \eqref{eq:divisibility}, so that $r_m(n)=|P_m(n)|$.
We begin this section by bifurcating $P_m(n)$ according
to whether or not a 1 appears as a part
in a given partition.  
Define
\begin{align*}
P_{m,1}(n)&=\{\pi\in P_m(n)\st
\alpha_1>0\},\\
P_{m,0}(n)&=\{\pi\in P_m(n)\st
\alpha_1=0\}
\end{align*}
so that $P_m(n)=P_{m,1}(n)\sqcup
P_{m,0}(n)$. 
We now divide $P_{m,0}(n)$ into
$2^{m+1}$
disjoint collections according to whether
$\alpha_i$ is congruent to 0 or
$2^{m+2-i}$ modulo $2^{m+3-i}$
for $2\leq i\leq m+2$.
Let $\epsilons$ be a finite
sequence of integers taking values
in $\{0,1\}$.
If we define
\[
P_{m,0}^{\epsilons}(n)=
\{\pi\in 
P_{m,0}(n)
\st\alpha_i\equiv \ve_i2^{m+2-i}
\mod 2^{m+3-i},
2\leq i\leq m+2\}
\]
then
\[
P_{m,0}(n)=\bigsqcup_{\ve_i\in\{0,1\}}
P_{m,0}^{\epsilons}.
\]
\begin{lemma}\label{shift}
The map $\sigma:
P_{m,0}^{\zeros}(n)\to 
P_m\left(\frac n2\right)$ defined by
\[
\sigma:(n;\zeroalphas)\mapsto
\left(\frac n2;\alpha_2,\alpha_3\ldots
\right)
\]
is a bijection.
\end{lemma}
\begin{proof}
A partition $\pi=(n;\zeroalphas)
\in P_{m,0}^{\zeros}(n)$ satisfies
$2^{m+3-i}\mid\alpha_i$ for
$2\leq i\leq m+2$.  Therefore
\[
\sigma(\pi)=\left(\frac n2;
\alpha_2,\alpha_3,\ldots\right)
=\left(\frac n2;\beta_1,\beta_2,
\ldots\right)
\]
where $2^{m+2-i}\mid
\alpha_{i+1}=\beta_i$
for $1\leq i\leq m+1$,
which implies
$\sigma(\pi)
\in P_m\left(\frac n2
\right)$.  If 
$\pi'=\left(
\frac n2;\alphas
\right)
\in P_m\left(\frac
n2\right)$ and we define
$\sigma^{-1}$ by
\[
\sigma^{-1}(\pi')=
(n;0,\alphas)=
(n;0,\gamma_2,\gamma_3,\ldots)
\]
we see that $2^{m+3-i}\mid
\alpha_{i-1}=\gamma_i$
for $2\leq i\leq m+2$.
Hence $\sigma^{-1}(\pi')\in
P_{m,0}^{\zeros}(n)$, 
and it follows from the
formulas for $\sigma$ and
$\sigma^{-1}$ that $\sigma\circ
\sigma^{-1}=1_{P_{m,0}^{\zeros}(n)}$
and
$\sigma^{-1}\circ\sigma
=1_{P_m\left(\frac
n2\right)}$.
\end{proof}
\begin{lemma}\label{taus}
If $\ve_i=1$ for some
$i$, $2\leq i\leq m+2$, 
the map 
\[
\tau_i:
P_{m,0}^{\ve_2,\ldots,
\ve_i,
\ldots,\ve_{m+2}}(n)\to
P_{m,0}^{\ve_2,\ldots,0,
\ldots,\ve_{m+2}}(n-2^{m+1})
\]
defined by
\[
\tau_i:(n;\zeroalphas)\mapsto
(n-2^{m+1};0,\alpha_2,\ldots,
\alpha_i-2^{m+2-i},\ldots)
\]
is a bijection.  
\end{lemma}
\begin{proof}
If $\ve_i=1$ and
$\pi=(n;\zeroalphas)\in 
P_{m,0}^{\ve_2,\ldots,
\ve_i,
\ldots,\ve_{m+2}}(n)$,
then $\ve_i=2^{m+2-i}
\mod 2^{m+3-i}$, 
which means 
$\alpha_i-2^{m+2-i}\equiv
0\mod 2^{m+3-i}$.
Hence $\tau_i(\pi)\in
P_{m,0}^{\ve_2,\ldots,0,
\ldots,\ve_{m+2}}
(n-2^{m+1})$.  If 
$\pi'=(n-2^{m+1};\zeroalphas)
\in P_{m,0}^{\ve_2,\ldots,0,
\ldots,\ve_{m+2}}
(n-2^{m+1})$ and we define
$\tau_i^{-1}$ by
\[
\tau_i^{-1}(\pi')=
(n;0,\alpha_2,\ldots,
\alpha_i+2^{m+2-i},\ldots)
=(n;0,\beta_2,\beta_3,\ldots)
\]
we see that $\beta_i=2^{m+2-i}\mod
2^{m+3-i}$.  Hence $\tau_i^{-1}(\pi')
\in P_{m,0}^{\ve_2,\ldots,
\ve_i,
\ldots,\ve_{m+2}}(n)$.  It follows
that $\tau_i\circ\tau_i^{-1}=
1_{P_{m,0}^{\ve_2,\ldots,
\ve_i,
\ldots,\ve_{m+2}}}$
and $\tau_i^{-1}\circ\tau_i=
1_{P_{m,0}^{\ve_2,\ldots,0,
\ldots,\ve_{m+2}}}$.
\end{proof}
\begin{lemma}\label{tauone}
The map $\tau_1:P_{m,1}(n)\to
P_m(n-2^{m+1})$ defined by
$\tau_1:(n;\alphas)\mapsto
(n-2^{m+1};\alpha_1-2^{m+1},
\alpha_2,\alpha_3,\ldots)$
is a bijection.
\end{lemma}
\begin{proof}
Let $\pi=(n;\alphas)\in P_{m,1}(n)$.
Then $2^{m+1}\alpha_1\geq1$, which
implies $\alpha_1\geq 2^{m+1}$.  
Hence $\tau_1(\pi)\in P_m(n-2^{m+1}$.
For $\pi'=(n-2^{m+1},\alphas)\in 
P_m(n-2^{m+1})$, we may define
$\tau_1^{-1}(\pi')=(n;
\alpha_1+2^{m+1},\alpha_2,\alpha_3,
\ldots)$.
\end{proof}
\begin{lemma}\label{oddeven} Suppose
$\ve_i=1$ for
indices
$2\leq i_1<
i_2<\cdots<i_k\leq m+1$
where $k\geq0$,
and $\ve_i=0$ for
all other $i$ such that
$2\leq i\leq m+1$.  
\begin{enumerate}
\item\label{firstcase} 
If either $n\equiv
2^{m+1}\mod 2^{m+2}$
and $k$ is even, or
$n\equiv
0\mod 2^{m+2}$
and $k$ is odd,
then 
$P_{m,0}^{\epsilons}(n)=
P_{m,0}^{\ve_2,\ldots,\ve_{m+1},
1}(n)$,
and the map
\[
\sigma\circ
\tau_{i_1}\circ
\tau_{i_2}\circ
\cdots\circ\tau_{i_k}\circ
\tau_{m+2}:P_{m,0}
^{\epsilons}\to P_m\left(
\frac{n-(k+1)2^{m+1}}2
\right)
\]
is a bijection.
\item\label{secondcase} 
If either $n\equiv
2^{m+1}\mod 2^{m+2}$
and $k$ is odd, or
$n\equiv
0\mod 2^{m+2}$
and $k$ is even,
then 
$P_{m,0}^{\epsilons}(n)=
P_{m,0}^{\ve_2,\ldots,\ve_{m+1},
0}(n)$,
and
the map
\[
\sigma\circ
\tau_{i_1}\circ
\tau_{i_2}\circ
\cdots\circ\tau_{i_k}:P_{m,0}
^{\epsilons}\to P_m\left(
\frac{n-k\cdot2^{m+1}}2
\right)
\]
is a bijection.
\end{enumerate}
\end{lemma}
\begin{proof}
Let $\pi=(n;\zeroalphas)
\in P_{m,0}^{\epsilons}(n)$.
Since 
$\displaystyle n=\sum_{i\geq1}
\alpha_i\cdot2^{i-1}$, we have
\[
n\equiv\sum_{i=1}^{m+2}\alpha_i
\cdot 2^{i-1}\mod 2^{m+2}.
\]
If $i\notin\{i_1,\ldots,i_k\}$, 
then $\alpha_i\cdot2^{i-1}\equiv
0\mod2^{m+2}$, whereas if
$i\in\{i_1,\ldots,i_k\}$, then
$\alpha_i\cdot2^{i-1}\equiv
2^{m+1}\mod2^{m+2}$.  
This implies
$2^{m+1}\equiv 
n\equiv\alpha_{i_1}\cdot2^{i_1-1}
+\cdots+\alpha_{i_k}\cdot2^{i_k-1}
+\alpha_{m+2}\cdot2^{m+1}
\equiv k\cdot2^{m+1}
+\alpha_{m+2}\cdot2^{m+1}
\mod 2^{m+2}$.
Thus, whether we assume that
$n\equiv2^{m+1}\mod 2^{m+2}$ and
$k$ is even, or that 
$n\equiv0\mod2^{m+2}$ and $k$
is odd, 
it follows that $\alpha_{m+2}$ must be
odd, i.e., $\ve_{m+2}$ must equal 1. 
An argument similar to the proof
of Lemma \ref{taus} shows that
the map $\tau_{m+2}:
P_{m,0}^{\ve_2,\ldots,\ve_{m+1},1}
(n)
\to P_{m,0}^{\ve_2,\ldots,\ve_{m+1},
0}(n-2^{m+1})$ defined by
$\tau_{m+2}:\pi\mapsto
(n-2^{m+1};0,\alpha_2,\ldots,
\alpha_{m+2}-1,\ldots)$
is a bijection.  Therefore,
by Lemmas \ref{shift} and
\ref{taus}, the map
$\sigma\circ\tau_{i_1}\circ
\tau_{i_2}\circ\cdots\circ
\tau_{i_k}\circ\tau_{m+2}$
is a bijection between the indicated
source and target.  This proves
\eqref{firstcase}. 

Note similarly that either pair of hypotheses
in \eqref{secondcase}
forces
$P_{m,0}^{\epsilons}(n)=
P_{m,0}^{\ve_1,\ldots,
\ve_{m+1},0}(n)$.  Therefore,
by Lemmas \ref{shift} and
\ref{taus}, the map
$\sigma\circ\tau_{i_1}\circ
\tau_{i_2}\circ\cdots\circ
\tau_{i_k}$ is a bijection between
the indicated source and target.
\end{proof}

\section{Proof of Theorem~\ref{thm:main}} \label{sec:proof_main_theorem}
In this section we will prove
Theorem \ref{thm:main}.  It is evident
that $r_m(n)=0$ if $2^{m+1}\nmid
n$.  Assume $2^{m+1}\mid n$.  
We established
in Section \ref{sec:preliminary_results} that
\[
P_m(n)=P_{m,1}(n)\sqcup
\bigsqcup_{i\in\{0,1\}}
P_{m,0}^{\epsilons}.
\]
Lemma \ref{tauone} implies
$P_{m,1}(n)$ and $P_m(n-2^{m+1})$
are in bijective correspondence.  
This contributes $r_m(n-2^{m+1})$
to the total size of $P_m(n)$.

It
remains to obtain similar
correspondences for the 
$2^{m+1}$ disjoint
sets $P_{m,0}^{\epsilons}(n)$ that 
comprise $P_{m,0}(n)$.  To do this,
we divide into the two cases
corresponding to the formulas
given in Theorem \ref{thm:main}: $n\equiv
2^{m+1}\mod2^{m+2}$ and
$n\equiv0\mod2^{m+2}$.  

Suppose first that 
$n\equiv
2^{m+1}\mod2^{m+2}$.  If exactly 
$2j$ of the $m-1$ superscripts
$\ve_2,\ldots,\ve_{m+1}$ are equal
to 1, where $0\leq j\leq\floor{m/2}$,
Lemma \ref{oddeven}\eqref{firstcase}
implies $P_{m,0}^{\epsilons}(n)$
and $P_m((n-(2j+1)2^{m+1})/2)$
are in bijective correspondence.
Therefore, each such $j$
contributes
\[
\binom{m}{2j}r_m\left(
\dfrac{n-(2j+1)2^{m+1}}2
\right)
\]
to the total size of $P_m(n)$.  
If exactly $2j+1$ of the superscripts
$\ve_2,\ldots,\ve_{m+1}$ are
equal to 1, where
$0\leq j\leq\floor{(m-1)/2}$, Lemma 
\ref{oddeven}\eqref{secondcase} 
implies 
$P_{m,0}^{\epsilons}(n)$ and 
$P_m((n-(2j+1)2^{m+1})/2)$
are in bijective correspondence.
Therefore, each such $j$ 
contributes
\[
\binom{m}{2j+1}r_m\left(
\dfrac{n-(2j+1)2^{m+1}}2
\right)
\]
to the total size of $P_m(n)$.  
Thus
\begin{align*}
r_m(n)&=r_m(n-2^{m+1})
+\sum_{j=0}^{\floor{m/2}}
\binom m{2j}
r_m\left(
\dfrac{n-(2j+1)2^{m+1}}2
\right)\\
&\hspace{1.5in}+\sum_{j=0}^{\floor{(m-1)/2}}
\binom m{2j+1}
r_m\left(
\dfrac{n-(2j+1)2^{m+1}}2
\right)
\\
&=r_m(n-2^{m+1})
+\sum_{j=0}^{\floor{m/2}}
\binom{m+1}{2j+1}
r_m\left(
\dfrac{n-(2j+1)2^{m+1}}2
\right)
\end{align*}
by Pascal's rule.

Next, suppose
$n\equiv0\mod2^{m+2}$.
If exactly $2j$ of the superscripts
$\ve_2,\ldots,\ve_{m+1}$ are equal to
1, where $0\leq j\leq
\floor{m/2}$,
Lemma \ref{oddeven}\eqref{secondcase}
implies $P_{0,m}^{\epsilons}(n)$
and $P_m((n-2j\cdot2^{m+1})/2)$ are in bijective 
correspondence.  Therefore, each such
$j$ contributes
\[
\binom m{2j}r_m\left(
\dfrac{n-2j\cdot2^{m+1}}2
\right)
\]
to the total size of $P_m(n)$.  If
exactly $2j-1$ of the superscripts
$\ve_2,\ldots,\ve_{m+1}$ are equal
to 1,
where $0<j\leq\floor{(m+1)/2}$,
Lemma \ref{oddeven}\eqref{firstcase}
implies $P_{0,m}^{\epsilons}(n)$
and $P_m((n-2j\cdot2^{m+1})/2)$ are in bijective
correspondence.  Therefore, each
such $j$ contributes
\[
\binom m{2j-1}r_m\left(
\dfrac{n-2j\cdot2^{m+1}}2
\right)
\]
to the total size of $P_m(n)$.  Thus
\begin{align*}
r_m(n)&=r_m(n-2^{m+1})
+\sum_{j=0}^{\floor{m/2}}
\binom m{2j}
r_m\left(
\dfrac{n-2j\cdot2^{m+1}}2
\right)
+\sum_{j=1}^{\floor{(m+1)/2}}
\binom m{2j-1}
r_m\left(
\dfrac{n-2j\cdot2^{m+1}}2
\right)\\
&=r_m(n-2^{m+1})
+\sum_{j=0}^{\floor{(m+1)/2}}
\binom{m+1}{2j}r_m\left(
\dfrac{n-2j\cdot2^{m+1}}2
\right)
\end{align*}
by Pascal's rule.

\section{Proof of Theorem~\ref{thm:main_partial}}\label{sec:partial_sum}
\label{partial}
In this section, we show that 
$r_m(n)$ is a
sequence of partial sums of 
$r_{m-1}(n)$ for $m\geq0$.  
It will be convenient to have
the following lemma, which identifies
two special cases of Theorem
\ref{thm:main}.
\begin{lemma}\label{sumlemma}
For $m\geq0$, $r_{m-1}(2^m)=m+1$
and $r_{m-1}(2^{m+1})=
2m+2+\binom m2$.
\end{lemma}
Let $\displaystyle S_m(n)=
\sum_{k=0}^{
n/2^{m+1}}r_{m-1}(2^{m}k)$.  
Then $S_m(0)=r_m(0)=1$.  
It therefore
suffices to show that $S_m(n)$ 
obeys the recursive formula
for $r_m(n)$ given
in Theorem \ref{thm:main}.  

Suppose $n\equiv 2^{m+1}\mod
2^{m+2}$.  Let $n=2^{m+2}n'+2^{m+1}$.
We shall use induction
on $n'$.  The base case $n'=0$ holds
since, by Lemma \ref{sumlemma},
\[
S_m(2^{m+1})=\sum_{k=0}^1r_{m-1}(2^mk)
=r_{m-1}(0)+r_{m-1}(2^m)=1+m+1=m+2
\]
while
\begin{align*}
S_m(2^{m+1}-2^{m+1})
+\sum_{j=0}^{\floor{m/2}}
\binom{m+1}{2j+1}S_m\left(
\frac{2^{m+1}-(2j+1)2^{m+1}}2
\right)
=(m+2)S_m(0)=m+2.
\end{align*}
Assume that,
for a fixed $n'\geq0$,
\[
S_m(2^{m+2}n'+2^{m+1})
=S_m(2^{m+2}n')+
\sum_{j=0}^{\floor{m/2}}
\binom{m+1}{2j+1}S_m\left(
\frac{
2^{m+2}n'+2^{m+1}
-(2j+1)2^{m+1}}2
\right).
\]
Then
\begin{align*}
S_m(2^{m+2}(n'+1)+2^{m+1})&=
\sum_{k=0}^{2n'+3}r_{m-1}
(2^mk)\\
&=
\underbrace{S_m(
2^{m+2}n'+2^{m+1})}_A
+\underbrace{r_{m-1}(
2^m(2n'+2))}_B
+\underbrace{r_{m-1}(
2^m(2n'+3))}_C.
\end{align*}
We know an expression for $A$
by hypothesis.  By Theorem
\ref{thm:main},
\begin{align*}
B&=r_{m-1}(2^m(2n'+2)-2^m)
+\sum_{j=0}^{\floor{m/2}}
\binom m{2j}r_{m-1}\left(
\dfrac{2^m
(2n'+2)-2j\cdot
2^m}2
\right)\\
&=r_{m-1}(2^m(2n'+1))
+\sum_{j=0}^{\floor{m/2}}
\binom m{2j}r_{m-1}(
2^m(n'-j+1)
)
\end{align*}
and
\begin{align*}
C&=r_{m-1}(2^m(2n'+3)-2^m)
+\sum_{j=0}^{\floor{(m-1)/2}}
\binom m{2j+1}r_{m-1}\left(
\frac{2^m(2n'+3)-(2j+1)2^m}2
\right)\\
&=r_{m-1}(2^m(2n'+2))
+\sum_{j=0}^{\floor{(m-1)/2}}
\binom m{2j+1}r_{m-1}
(
2^m(n'-j+1)
),
\end{align*}
which implies
\begin{align*}
B+C&=\sum_{k=2n'+1}^{2n'+2}
r_{m-1}(2^mk)
+\sum_{j=0}^{\floor{m/2}}
\left(
\binom m{2j}+\binom m{2j+1}
\right)r_m(2^m(n'-j+1))\\
&=\sum_{k=2n'+1}^{2n'+2}
r_{m-1}(2^mk)
+\sum_{j=0}^{\floor{m/2}}
\binom{m+1}{2j+1}
r_m(2^m(n'-j+1)).
\end{align*}
It follows from the induction
hypothesis that
\begin{align*}
S_m(2^{m+2}(n'+1)+2^{m+1})
&=A+B+C\\
&=S_m(2^{m+2}n'+2^{m+2})
+\sum_{j=0}
^{\floor{m/2}}
\binom{m+1}{2j+1}
S_m(2^{m+1}(n'-j+1))\\
&=S_m(2^{m+2}(n'+1))\\
&\hspace{0.5in}+\sum_{j=0}
^{\floor{m/2}}
\binom{m+1}{2j+1}
S_m\left(
\dfrac{2^{m+2}(n'+1)+2^{m+1}
-(2j+1)2^{m+1}}
2
\right).
\end{align*}

Next, suppose $n\equiv0\mod
2^{m+2}$.  Let $n=2^{m+2}n'$.
We shall use induction
on $n'$.  The base case
$n'=1$ holds since,
by Lemma \ref{sumlemma},
\[
S_m(2^{m+2})=\sum_{k=0}^2
r_{m-1}(2^mk)=r_{m-1}(0)
+r_{m-1}(2^m)+r_{m-1}(2^{m+1})
=3m+4+\binom m2
\]
while 
\begin{align*}
S_m(2^{m+2}&-2^{m+1})+
\sum_{j=0}^{\floor{(m+1)/2}}
\binom{m+1}{2j}S_m\left(
\dfrac{2^{m+2}-2j\cdot2^{m+1}}2
\right)\\
&=S_m(2^{m+1})+
\sum_{j=0}^{\floor{(m+1)/2}}
\binom{m+1}{2j}S_m(2^{m+1}(1-j))
\\
&=2S_m(2^{m+1})
+\binom{m+1}2S_m(0)\\
&=2(r_{m-1}(0)+r_{m-1}(2^m))
+\binom{m+1}2\\
&=2m+4+\binom{m+1}2=3m+4
+\binom m2.  
\end{align*}
Assume that, 
for a fixed $n'\geq1$,
\[
S_m(2^{m+2}n')=S_m(2^{m+2}n'
-2^{m+1})+\sum_{j=0}
^{\floor{(m+1)/2}}\binom{m+1}{2j}
S_m\left(
\dfrac{2^{m+2}n'-2j\cdot2^{m+1}}2
\right).
\]
Then
\begin{align*}
S_m(2^{m+2}(n'+1))&=
\sum_{k=0}^{2n'+2}r_{m-1}(2^mk)\\
&=\underbrace{S_m(2^{m+2}n')}_A
+\underbrace{r_{m-1}
(2^m(2n'+1))}_B
+\underbrace{r_{m-1}
(2^m(2n'+2))}_C.
\end{align*}
We know an expression
for $A$ by hypothesis.
By Theorem \ref{thm:main},
\begin{align*}
B&=r_{m-1}(2^m(2n'+1)-2^m)
+\sum_{j=0}
^{\floor{(m-1)/2}}\binom m{2j+1}
r_{m-1}\left(
\dfrac{2^m(2n'+1)-(2j+1)2^m}2
\right)\\
&=r_{m-1}(2^m(2n'))
+\sum_{j=0}
^{\floor{(m-1)/2}}\binom m{2j+1}
r_{m-1}(2^m(n'-j))
\end{align*}
and
\begin{align*}
C&=r_{m-1}(2^m(2n'+2)-2^m)
+\sum_{j=0}
^{\floor{m/2}}\binom m{2j}
r_{m-1}\left(
\dfrac{2^m(2n'+2)-2j\cdot2^m}2
\right)\\
&=r_{m-1}(2^m(2n'+1))
+\sum_{j=0}
^{\floor{m/2}}\binom m{2j}
r_{m-1}(2^m(n'-j+1)).
\end{align*}
After a change of indexing
variable in the sum 
appearing in $B$, we obtain
\begin{align*}
B+C&=\sum_{k=2n'}^{2n'+1}
r_{m-1}(2^mk)+\sum_{j=0}
^{\floor{(m+1)/2}}
\left(
\binom{m}{2j-1}+
\binom{m}{2j}
\right)
r_{m-1}(2^m(n'-j+1))\\
&=\sum_{k=2n'}^{2n'+1}
r_{m-1}(2^mk)+
\sum_{j=0}
^{\floor{(m+1)/2}}
\binom{m+1}{2j}
r_{m-1}(2^m(n'-j+1)).
\end{align*}
It follows from the
induction hypothesis that
\begin{align*}
S_m(2^{m+2}(n'+1))&=A+B+C\\
&=S_m(2^{m+2}n'+2^{m+1})
+\sum_{j=0}
^{\floor{(m+1)/2}}
\binom{m+1}{2j}
S_m(2^{m+1}
(n-j+1))\\
&=S_m(2^{m+2}(n'+1)-2^{m+1})
+\sum_{j=0}
^{\floor{(m+1)/2}}
\binom{m+1}{2j}
S_m\left(
\dfrac{2^{m+2}(n'+1)-2j
\cdot2^{m+1}}2
\right).
\end{align*}
This concludes the proof of Theorem \ref{thm:main_partial}.

\bibliographystyle{alpha}
\bibliography{blarson}

\vskip.3truein

\noindent{\em Department of Mathematics, Clayton State University}\\
\noindent{\em 2000 Clayton State Boulevard,
Morrow, GA 30260}\vskip.1truein

\noindent{\em Department of Mathematics, The Catholic University of America}\\
\noindent{\em 620 Michigan Ave NE, Washington DC 20064}\vskip.1truein

\noindent\texttt{scottbailey@clayton.edu,}\quad\texttt{larsond@cua.edu}

\end{document}

%% file: milgram.tex
\begin{tikzpicture}[scale=0.80]
\foreach \i in {0,1,2,3}{
    \foreach \dim in {0,2,3,4,5,6,10,11,12}{
        \draw (\dim, -4*\i) circle [radius=2pt, fill=white];
    };
    \foreach \dim in {0,1,2,3,...,12}{
        \node (Q\i\dim) at (\dim, -4*\i) {};
    };
};

\foreach \i in {0,1,2,3}{
    \node (Qell\i) at (8, -4*\i) {$\cdots$};
}

\draw (13, -4) circle [radius=2pt, fill=white];
\node (Q113) at (13, -4) {};
\draw (14, -4) circle [radius=2pt, fill=white];
\node (Q114) at (14, -4) {};

\draw (9, -7) circle [radius=2pt, fill=white];
\node (Q29a) at (9, -7) {};
\draw (10, -7) circle [radius=2pt, fill=white];
\node (Q210a) at (10, -7) {};
\draw (12, -7) circle [radius=2pt, fill=white];
\node (Q212a) at (12, -7) {};
\draw (13, -7) circle [radius=2pt, fill=white];
\node (Q213a) at (13, -7) {};

\draw (8, -11) circle [radius=2pt, fill=white];
\node (Q38a) at (8, -11) {};
\draw (9, -11) circle [radius=2pt, fill=white];
\node (Q39a) at (9, -11) {};

\foreach \i in {0,1,2,3}{
        \draw (Q\i0) to[out=-45,in=-135] (Q\i2);
        \draw (Q\i4) to[out=-45,in=-135] (Q\i6);
        \draw (Q\i10) to[out=-45,in=-135] (Q\i12);
        \draw (Q\i3) to[out=45, in=135] (Q\i5);
        \draw (Q\i2) -- (Q\i3);
        \draw (Q\i4) -- (Q\i5);
        \draw (Q\i6) -- (Q\i7);
        \draw (Q\i9) -- (Q\i10);
        \draw (Q\i11) -- (Q\i12);
        
};

\draw (Q113) -- (Q114);
\draw (Q111) to[out=45, in =135] (Q113);

\draw (Q29a) -- (Q210a);
\draw (Q212a) -- (Q213a);
\draw (Q29a) to[out=-90, in=90] (Q211);
\draw (Q211) to[out=90, in=-90] (Q213a);
\draw (Q210a) to[out=45, in=135] (Q212a);

\draw (Q38a) -- (Q39a);
\draw (Q38a) to[out=-60, in=120] (Q310);
\draw (Q39a) to[out=-60, in=120] (Q311);

\foreach \i in {0, 1, 2, 3}{
    \node (label\i) at (-1,-4*\i) {$Q_{\i,n}$};
};

\node (labeldim0) at (0,1.25) {$0$};
\node (labeldim4) at (4,1.25) {$4$};
\node (labeldim4n) at (11,1.25) {$4n$};

\draw[dash dot] (0,1)--(0,-13);
\draw[dash dot] (4,1)--(4,-13);
\draw[dash dot] (11,1)--(11,-13);

\end{tikzpicture}